\documentclass[11pt,a4paper]{article}

\usepackage[T1]{fontenc}
\usepackage[utf8]{inputenc}
\usepackage{amsmath,amssymb,amsthm}
\usepackage{mathtools}
\usepackage[margin=2.6cm]{geometry}
\usepackage{tikz}
\usepackage{hyperref}
\hypersetup{hidelinks}

\theoremstyle{plain}
\newtheorem{theorem}{Theorem}[section]
\newtheorem{lemma}[theorem]{Lemma}
\newtheorem{proposition}[theorem]{Proposition}
\newtheorem{corollary}[theorem]{Corollary}
\newtheorem{conjecture}[theorem]{Conjecture}
\theoremstyle{definition}
\newtheorem{definition}[theorem]{Definition}

\theoremstyle{remark}

\newtheorem{question}[theorem]{Question}

\newcommand{\Sym}{\operatorname{Sym}}
\newcommand{\Cay}{\operatorname{Cay}}
\newcommand{\Stab}{\operatorname{Stab}}

\newcommand{\id}{\mathbf{1}}

\newcommand{\coxrow}[2]{%
\begin{tikzpicture}[baseline=-0.5ex, x=0.62cm]
  \pgfmathtruncatemacro{\m}{#1-1}%
  \foreach \i in {1,...,\m}{%
    \pgfmathtruncatemacro{\j}{\i+1}%
    \ifnum\i<\m \draw[line width=0.5pt] (\i,0)--(\j,0);\fi
  }%
  \foreach \i in {1,...,\m}{%
    \ifnum#2=1
      \pgfmathtruncatemacro{\f}{mod(\m-\i,2)}%
    \else
      \pgfmathtruncatemacro{\f}{0}%
    \fi
    \ifnum\f=0 \fill (\i,0) circle (2.1pt);
    \else \draw[fill=white,line width=0.5pt] (\i,0) circle (2.1pt);\fi
  }%
\end{tikzpicture}}

\title{Genlex Gray codes for $S_n$ with the fewest operations:\\
classification and symmetry}
\author{Yehonathan Sharvit\\[2pt]
{\small Independent researcher \quad \texttt{viebel@gmail.com}}}
\date{August 2026}

\begin{document}
\maketitle

\begin{abstract}
A Gray code for $S_n$ is \emph{genlex} when words sharing a suffix are
consecutive. We determine the genlex Gray codes for $S_n$ that use the fewest
possible operations: Zaks' recursion generalises to a superfactorial family of
such codes, and no code outside this family attains the minimum. Every code
in the family closes into a cycle, and the cycle is invariant under a group of
left translations, cyclic of order $n$ or dihedral of order $2n$ according to
an explicit criterion on the operations.
In the pancake graph the family reduces to a single member, the classical Zaks
order; with respect to the standard parabolic chain of $W(A_{n-1})$, that order
attains one extreme of total Coxeter length.
\end{abstract}

\section{Introduction}

\subsection*{The problem}

A \emph{Gray code} for the symmetric group $S_n$ lists the $n!$ permutations of
$\{1,\dots,n\}$, written in one-line notation, so that consecutive words differ
by an operation drawn from a fixed set; M\"utze \cite{mutze} surveys the
subject. Following Walsh \cite{walsh}, a listing is \emph{genlex} when words
agreeing in their final letters are consecutive. The operations are the
generators of the Cayley graph in which the code is a walk, and we use the two
terms interchangeably in that sense.

This paper determines the genlex Gray codes for $S_n$ that use the fewest
distinct operations. The minimum is $n-1$ (Lemma~\ref{lem:budget}), and no
closure hypothesis is made: a genlex Gray code attaining the minimum
necessarily closes into a cycle (Corollary~\ref{cor:closure}). The
resulting family is superfactorially large yet completely rigid --- each member
is determined by one free choice at each level of a subgroup chain --- and its
two ends, ordered by total Coxeter length, are occupied by classical objects.

\subsection*{Results}

Let $G_k=\Sym\{1,\dots,k\}$, so that
$\{1\}=G_1<G_2<\cdots<G_n=S_n$ is the chain of standard parabolic subgroups of
the Coxeter group $W(A_{n-1})=S_n$ along a maximal flag. A \emph{tower}
(Definition~\ref{def:tower}) is a choice of one generator per layer,
$g_k\in G_k\setminus G_{k-1}$, whose successive quotients rotate the cosets at
each level.

\begin{itemize}
\item \textbf{Classification} (Theorem~\ref{thm:class}). Genlex Hamiltonian
cycles on $S_n$ with $n-1$ generators are in bijection with towers
(Definition~\ref{def:tower}). There are
$1!\,2!\cdots(n-1)!$ of them, one for each independent choice of a $k$-cycle at
each level; no compatibility condition between levels arises. Neither the
cyclicity of the listing nor that of the transversals is assumed: both are
conclusions.

\item \textbf{Structure} (\S\ref{sec:structure}). Zaks' recursion, read over the
alphabet $g_2,\dots,g_n$, closes into a cycle for every tower; the reason is a
compensation lemma. The resulting word is a palindrome, and ranking is the plain
factorial-base monomial $u_n^{c_n}\cdots u_2^{c_2}$, with no reflected
correction.

\item \textbf{Symmetry} (\S\ref{sec:symmetry}). The left-regular stabilizer of
the cycle is cyclic of order $n$ or dihedral of order $2n$, the second case
occurring exactly when every $g_k$ is an involution. In the language of Gregor,
Merino and M\"utze \cite{gmm}, every tower cycle is $n$-symmetric.

\item \textbf{Pancake} (\S\ref{sec:pancake}). Inside the pancake graph the
family collapses: Zaks' cycle is the unique genlex Hamiltonian cycle up to left
translation.

\item \textbf{Extremes of Coxeter length} (\S\ref{sec:coxeter}). Ordered by total
Coxeter length $\Lambda=\sum\ell(g_k)$, the family has Zaks' tower at the
maximum and, conjecturally, a tower built from the bipartite factors of a
Coxeter element at the minimum.

\end{itemize}

\subsection*{Antecedents}
\label{sec:antecedents}

The chain and its coset representatives are standard, and the wider setting is
surveyed in \cite{mutze}. In the language of Sims tables \cite{sims} and Knuth's
Algorithm~G \cite[\S7.2.1.2]{knuth4a}, a tower is exactly a \emph{cyclic} Sims
table --- one full cycle per level in place of an arbitrary transversal --- and
in that framework the forward half of Theorem~\ref{thm:class} is a
specialization: the listing exists, and the ranking monomial is the usual
mixed-radix decomposition along the chain. The same scheme was reached
independently by Kokosi\'nski \cite{kokosinski}, from Hall--Paige's coset
partitions and transversal theory rather than from Sims: choice functions along
the chain, a lexicographic order on them, and ranking and unranking for an
arbitrary choice of transversal --- \emph{any partition rule}, in his phrase.
His transversals are the transposition transversals of Hall--Paige, so his
listings are neither genlex nor Gray codes; the criterion of the present paper
selects, inside his space of partition rules, exactly the cyclic ones.

The contributions of the present paper are of a different nature: the converse
of the classification, the closure into a cycle, and the symmetry and Coxeter
statements, none of which concerns generation procedures. For Zaks' tower
itself, the dihedral left-regular stabilizer was determined in
\cite{sharvitbenjo}, together with its identification inside the full
automorphism group of the pancake graph; the present paper extends the
left-regular statement to the entire family. That the genlex
hypothesis is essential can be seen from Ehrlich's swap method
\cite[Algorithm~E]{knuth4a}, \cite{ehrlich}, which spends exactly one generator
per layer --- the star transpositions --- and is not genlex.

What Theorem~\ref{thm:class} calls Zaks' order, the case $g_k=r_k$, has a longer
history than that name suggests. It is the algorithm of Ord-Smith
\cite{ordsmith}, whose listing agrees with Zaks' permutation by permutation, and
Knuth traces it further back still, to Kl\"ugel in 1796
\cite[\S7.2.1.7]{knuth4a}. The order predates the combinatorial literature:
the letter-permutation rules that Abraham Abulafia sets out in his
\emph{Or ha-Sekhel} in the thirteenth century generate this same sequence for
every $n$ \cite{abulafia}. We retain the name ``Zaks' tower'' for
definiteness; what is due to Zaks \cite{zaks} in this circle of ideas is the
$O(1)$ successor rule and the recursion in the form used below.

\section{Preliminaries}
\label{sec:prelim}

Permutations are written in one-line notation, $\pi=\pi_1\pi_2\cdots\pi_n$, and
operations act on \emph{positions}, on the right:
\[
  (\pi\cdot g)_i=\pi_{g(i)}.
\]
We write $\ell(\pi)$ for the Coxeter length of $\pi$ in $W(A_{n-1})=S_n$ with
simple reflections $s_i=(i\ \, i{+}1)$; as usual $\ell(\pi)=\operatorname{inv}(\pi)$,
the number of inversions. For a subset $J$ of simple reflections, $W_J$ is the
corresponding standard parabolic subgroup and $w_{0,J}$ its longest element. We
take $J_k=\{s_1,\dots,s_{k-1}\}$, so that
\[
  W_{J_k}=G_k=\Sym\{1,\dots,k\},\qquad
  \{1\}=G_1<G_2<\cdots<G_n=S_n ,
\]
a maximal chain of standard parabolic subgroups, with $[G_k:G_{k-1}]=k$.

\begin{definition}
A listing of $S_n$ is \emph{genlex} when, for every $k$, the words agreeing in
their last $n-k$ letters are consecutive. For a cyclic listing, read
``consecutive'' as ``occupying a contiguous arc''.
\end{definition}

The set of words agreeing with $\pi$ in positions $k+1,\dots,n$ is exactly the
coset $\pi G_k$, so the genlex condition is a condition on the chain above: each
of its cosets must be an arc. We call $\pi G_k$ a \emph{$k$-block}.

\begin{lemma}[minimum number of operations]
\label{lem:budget}
The layers $G_k\setminus G_{k-1}$, $k=2,\dots,n$, partition $S_n\setminus\{1\}$.
A genlex Hamiltonian listing of $S_n$, realised as a walk in a Cayley graph,
uses a generator from each layer, hence at least $n-1$ generators.
\end{lemma}

\begin{proof}
Every $\pi\neq\id$ lies in exactly one $G_k\setminus G_{k-1}$, namely for the
least $k$ with $\pi\in G_k$; this gives the partition. Fix
$k\in\{2,\dots,n\}$. By genlexity a $k$-block consists of consecutive words and
is partitioned into $k\ge2$ consecutive $(k-1)$-blocks, so some step inside the
$k$-block joins two distinct $(k-1)$-blocks. That step joins two words of the
same coset $\pi G_k$, so its generator lies in $G_k$; and it changes the
$(k-1)$-block, so its generator lies outside $G_{k-1}$. Hence a generator of the
layer $G_k\setminus G_{k-1}$ is used, for every $k$.
\end{proof}

Since $G_{k-1}$ is the stabilizer of the point $k$ in $G_k$, a coset
$xG_{k-1}$ of $G_{k-1}$ in $G_k$ is determined by the image $x(k)$; the $k$
cosets thus correspond to the $k$ possible letters in position $k$.

\begin{lemma}[coset rotation]
\label{lem:rotcoset}
For $u\in G_k$, the cyclic group $\langle u\rangle$ acts simply transitively on
the $k$ cosets of $G_{k-1}$ in $G_k$ if and only if $u$ is a $k$-cycle on
$\{1,\dots,k\}$.
\end{lemma}

\begin{proof}
The cosets $u^{\,j}G_{k-1}$, $j=0,\dots,k-1$, are pairwise distinct if and only
if the images $u^{\,j}(k)$ are, that is, if and only if the
$\langle u\rangle$-orbit of $k$ has size $k$; for $u\in\Sym\{1,\dots,k\}$ this
holds exactly when $u$ is a single $k$-cycle.
\end{proof}

\begin{definition}
\label{def:tower}
A \emph{tower} is a sequence $(g_2,\dots,g_n)$ with
$g_k\in G_k\setminus G_{k-1}$ whose successive quotients
$u_k:=g_{k-1}^{-1}g_k$ (with $g_1:=\id$) rotate the cosets at each level, that
is, each $u_k$ is a $k$-cycle on $\{1,\dots,k\}$.
\end{definition}

Equivalently a tower is the data of a $k$-cycle $u_k$ for each $k$, from which
$g_k=g_{k-1}u_k=u_2u_3\cdots u_k$. Theorem~\ref{thm:class}(ii) will show that
the $k$-cycle condition costs no generality: it is forced on every genlex
Hamiltonian listing with $n-1$ generators.

\begin{lemma}[levels are independent]
\label{lem:independent}
For any choice of $k$-cycles $u_k$ on $\{1,\dots,k\}$, $k=2,\dots,n$, the
elements $g_k=u_2\cdots u_k$ satisfy $g_k\in G_k\setminus G_{k-1}$. Consequently
the number of towers is the superfactorial $1!\,2!\cdots(n-1)!$.
\end{lemma}

\begin{proof}
Clearly $g_k\in G_k$. For $g_k\notin G_{k-1}$ it suffices that $g_k(k)\neq k$:
we have $g_k(k)=g_{k-1}(u_k(k))$, and $u_k(k)\in\{1,\dots,k-1\}$ because $u_k$
is a $k$-cycle, while $g_{k-1}$ preserves $\{1,\dots,k-1\}$. Hence
$g_k(k)\le k-1$. There are $(k-1)!$ $k$-cycles on $\{1,\dots,k\}$ and no
constraint links the levels, whence the count.
\end{proof}

Notably, no compatibility condition between consecutive levels arises.

\section{The classification}
\label{sec:class}

Given a tower, define words over the alphabet $\{g_2,\dots,g_n\}$ by Zaks'
recursion
\begin{equation}
\label{eq:rec}
  M_1=\varepsilon,\qquad
  M_k=\bigl(M_{k-1}\,g_k\bigr)^{k-1}M_{k-1}\quad(2\le k\le n).
\end{equation}
Thus $|M_k|=k!-1$. We read a word as a walk in $\Cay(S_n,\{g_2,\dots,g_n\})$
starting at a given vertex, each letter being a right multiplication, and write
$\operatorname{ev}(w)$ for the product of its letters.

\begin{lemma}[compensation]
\label{lem:comp}
$\operatorname{ev}(M_k)=g_k^{-1}$ for $1\le k\le n$. In particular
$\operatorname{ev}(M_n g_n)=\id$.
\end{lemma}

\begin{proof}
Induction on $k$. For $k=1$, $\operatorname{ev}(\varepsilon)=\id=g_1^{-1}$.
Assume $\operatorname{ev}(M_{k-1})=g_{k-1}^{-1}$. Then
$\operatorname{ev}(M_{k-1}g_k)=g_{k-1}^{-1}g_k=u_k$, so by \eqref{eq:rec}
\[
  \operatorname{ev}(M_k)=u_k^{\,k-1}\,g_{k-1}^{-1}
  =u_k^{-1}g_{k-1}^{-1}
  =(g_{k-1}u_k)^{-1}=g_k^{-1},
\]
using $u_k^{\,k}=\id$.
\end{proof}

Thus traversing a block undoes the generator that separates consecutive
blocks; this is what closes the listing, for every tower and with no reference
to prefix reversals.

\begin{lemma}[palindromicity]
\label{lem:palin}
$M_k$ is a palindrome in the letters $g_2,\dots,g_k$.
\end{lemma}

\begin{proof}
Induction. $M_1=\varepsilon$. If $M_{k-1}$ is a palindrome then, writing
$M_k=M_{k-1}g_kM_{k-1}g_k\cdots g_kM_{k-1}$ ($k$ occurrences of $M_{k-1}$
separated by $k-1$ letters $g_k$), reversing the word exchanges the blocks
pairwise and reverses each of them, which leaves the word unchanged.
\end{proof}

\begin{theorem}[classification]
\label{thm:class}
Let $n\ge3$.
\begin{enumerate}
\item[(i)] For every tower $(g_2,\dots,g_n)$, the word $M_n g_n$ traverses a
genlex Hamiltonian cycle in $\Cay(S_n,\{g_2,\dots,g_n\})$.
\item[(ii)] Conversely, every genlex Hamiltonian \emph{listing} of $S_n$ using
$n-1$ generators arises this way from a unique tower: its generator in layer
$k$ is $g_k$, and the quotients $g_{k-1}^{-1}g_k$ are necessarily $k$-cycles.
\end{enumerate}
Zaks' order is the case $g_k=r_k$, the reversal of the first $k$ letters.
\end{theorem}

\begin{proof}[Proof of (i)]
We show by induction on $k$ that, started at any vertex $\pi$, the word $M_k$
visits the $k!$ elements of the block $\pi G_k$, each exactly once, and ends at
$\pi g_k^{-1}$.

For $k=1$ this is vacuous. Assume it for $k-1$. By \eqref{eq:rec} the walk
consists of $k$ traversals of $M_{k-1}$ separated by $k-1$ steps $g_k$. By the
induction hypothesis the $j$-th traversal covers a full $(k-1)$-block, and by
Lemma~\ref{lem:comp} the vertex at which the $j$-th traversal begins is
$\pi\,(g_{k-1}^{-1}g_k)^{\,j}=\pi u_k^{\,j}$ for $j=0,\dots,k-1$. By
Lemma~\ref{lem:rotcoset} the group $\langle u_k\rangle$ acts simply
transitively on the $k$ cosets of $G_{k-1}$ in $G_k$; hence the $k$ blocks
$\pi u_k^{\,j}G_{k-1}$ are pairwise distinct and exhaust $\pi G_k$. Each
vertex of $\pi G_k$ is therefore visited exactly once, and the walk ends at
$\pi u_k^{\,k-1}g_{k-1}^{-1}=\pi g_k^{-1}$, as claimed.

Taking $k=n$ and $\pi=\id$, the walk $M_n$ visits all of $S_n$ and ends at
$g_n^{-1}$; the final letter $g_n$ returns to $\id$, so $M_ng_n$ is a
Hamiltonian cycle. Genlexity is the statement just proved: for each $k$ every
$k$-block is covered by one contiguous traversal of $M_k$.
\end{proof}

\begin{proof}[Proof of (ii)]
Let $C$ be a genlex Hamiltonian cycle using generators $h_2,\dots,h_n$, one per
layer, $h_k\in G_k\setminus G_{k-1}$ (Lemma~\ref{lem:budget}).

Fix $k$ and consider a $k$-block $B=\pi G_k$. By genlexity $B$ is an arc, and
$B$ is partitioned into $k$ sub-arcs, the $(k-1)$-blocks it contains. A step
inside a $k$-block uses a generator of $G_k$, that is, one of
$h_2,\dots,h_k$; a step crossing from one $(k-1)$-block to the next uses a
generator outside $G_{k-1}$, that is, one of $h_k,\dots,h_n$. A step doing both
therefore uses $h_k$. Hence all $k-1$ internal boundaries of $B$ are traversed
by the letter $h_k$, and the letter sequence of $C$ is forced to be the ruler
word \eqref{eq:rec} over the alphabet $h_2,\dots,h_n$; that is, $C$ is traversed
by $M_ng_n$ with $g_k=h_k$.

The argument uses only the internal block boundaries, so it applies verbatim to
a Hamiltonian path: a listing of $n!$ words has $n!-1$ steps, which is the
length of $M_n$.

It remains to see that each $u_k=h_{k-1}^{-1}h_k$ is a $k$-cycle. Run the
argument of part (i) in reverse: the $k$ sub-blocks of a $k$-block begin at the
vertices $\pi u_k^{\,j}$, $j=0,\dots,k-1$, and Hamiltonicity forces these to lie
in $k$ distinct cosets of $G_{k-1}$. Thus $\langle u_k\rangle$ acts simply
transitively on the $k$ cosets, and $u_k$ is a $k$-cycle by
Lemma~\ref{lem:rotcoset}.
\end{proof}

\begin{corollary}[closure is automatic]
\label{cor:closure}
A genlex Hamiltonian listing of $S_n$ using $n-1$ generators, started at $\id$,
ends at $g_n^{-1}$; appending the single letter $g_n$ closes it into a
Hamiltonian cycle.
\end{corollary}

\begin{proof}
By Theorem~\ref{thm:class}(ii) the listing is traversed by $M_n$, and
$\operatorname{ev}(M_n)=g_n^{-1}$ by Lemma~\ref{lem:comp}.
\end{proof}

\begin{corollary}[distinct towers, distinct cycles]
\label{cor:distinct}
Distinct towers give distinct Hamiltonian cycles.
\end{corollary}

\begin{proof}
From the edge set of $C$, the two neighbours of $\id$ are $g_2$ and $g_n^{-1}$,
which lie in different layers for $n\ge3$; so the two directions of traversal
are distinguishable and the letter sequence is determined by $C$. The vertex at
position $(k-1)!$ is then $u_k$, which recovers the tower.
\end{proof}

\section{Ranking}
\label{sec:structure}

\begin{proposition}[ranking]
\label{prop:rank}
Let $(g_2,\dots,g_n)$ be a tower and $C$ its cycle, started at $\id$. For
$0\le i<n!$ write $i=\sum_{k=2}^{n}c_k\,(k-1)!$ with $0\le c_k<k$, the
factorial-base expansion. Then the vertex at position $i$ is
\[
  \pi(i)=u_n^{c_n}u_{n-1}^{c_{n-1}}\cdots u_2^{c_2}.
\]
The generator applied between positions $i$ and $i+1$ is $g_k$, where $k$ is the
least level with $c_k<k-1$; equivalently, the largest $k$ with $(k-1)!\mid i+1$.
\end{proposition}

\begin{proof}
Immediate from the proof of Theorem~\ref{thm:class}(i): the walk enters the
$c_n$-th $(n-1)$-block at $u_n^{c_n}$, and recursively within it the $c_{n-1}$-th
$(n-2)$-block at $u_n^{c_n}u_{n-1}^{c_{n-1}}$, and so on. The statement about
the generator is the carry rule of the mixed-radix counter, which by
\eqref{eq:rec} is the position of the letter of level $k$ in the ruler word.
\end{proof}

Two features of Proposition~\ref{prop:rank} merit comment. First, there is
\emph{no reflected correction}:
in the reflected Gray code and in Trotter--Johnson the digit must be replaced by
its complement according to the parity of the higher digits, because the
recursion is boustrophedon; here \eqref{eq:rec} is repetitive, each block being
traversed in the same direction, so the digits are mutually independent and
ranking and unranking are $O(n)$ with no parity bookkeeping. Second, the normal
form of Proposition~\ref{prop:rank} has the shape of a \emph{polycyclic} normal
form --- bounded powers of $n-1$ fixed elements along a chain --- although $S_n$
is not polycyclic and the chain $G_k$ is not subnormal. What is lost without
normality is collection: the coordinates of a product are not a function of the
coordinates of the factors.

\section{Symmetry}
\label{sec:symmetry}

Let $C$ be the cycle of a tower, viewed as a set of edges. Edges are right
multiplications, so left translations $L_g:\sigma\mapsto g\sigma$ are graph
automorphisms; $\Stab(C)$ denotes the setwise stabilizer of $C$ inside the
left-regular copy $L(S_n)$. We claim nothing about the full automorphism group
of the graph; for Zaks' cycle in the pancake graph, the passage from the
left-regular stabilizer to the full automorphism group is carried out in
\cite{sharvitbenjo}. Since the left-regular action is free, $\Stab(C)$ embeds without
fixed points in the dihedral group of a cycle of length $n!$, hence is cyclic or
dihedral, and only its order is at issue.

The symmetry at issue is the one Gregor, Merino and M\"utze call
\emph{compression} \cite{gmm}: a Hamiltonian cycle on $N$ vertices is
$k$-symmetric when rotating it by $N/k$ positions is a graph automorphism.

\begin{theorem}[cyclic-order rigidity]
\label{thm:rigidity}
Let $n\ge3$, let $C$ be any cyclic ordering of $S_n$, and let $c\in S_n$ be an
$n$-cycle with $L_c(C)=C$. Then
\[
  \Stab_{L(S_n)}(C)\subseteq\{L_g:\ gcg^{-1}\in\{c,c^{-1}\}\},
\]
and in particular $|\Stab_{L(S_n)}(C)|\le 2n$.
\end{theorem}

\begin{proof}
The automorphism group of a cycle is dihedral, and $L_c$ acts on $C$ as a
rotation of order $n$ (it is fixed-point-free of order $n$). Any $L_g$
stabilizing $C$ conjugates that rotation to itself or to its inverse, so
$L_gL_cL_g^{-1}\in\{L_c,L_c^{-1}\}$; the left-regular action being faithful,
$gcg^{-1}\in\{c,c^{-1}\}$. The centralizer of an $n$-cycle in $S_n$ is
$\langle c\rangle$, of order $n$; the solutions of $gcg^{-1}=c^{-1}$ form a
single coset of it. Hence at most $2n$ elements.
\end{proof}

\begin{definition}
A Hamiltonian cycle $C$ in a Cayley graph on $S_n$ is \emph{cyclic} when
$\Stab(C)\cong C_n$ and \emph{dihedral} when $\Stab(C)\cong D_n$.
\end{definition}

\begin{lemma}[the rotation]
\label{lem:rotation}
For every tower, $L_{u_n}(C)=C$, and $L_{u_n}$ acts on $C$ as a rotation by
$(n-1)!$ positions, of order $n$. In particular every tower cycle is
$n$-symmetric and $\Stab(C)\supseteq\langle L_{u_n}\rangle\cong C_n$.
\end{lemma}

\begin{proof}
By \eqref{eq:rec}, $M_ng_n=(M_{n-1}g_n)^n$: the traversal is $n$ repetitions of
one word $B=M_{n-1}g_n$ of length $(n-1)!$, and $\operatorname{ev}(B)=u_n$ by
Lemma~\ref{lem:comp}. Hence the vertex at position $j(n-1)!+m$ equals
$u_n^{\,j}\cdot p_m$, where $p_m$ is the vertex at position $m$ of the first
copy. Advancing by $(n-1)!$ positions is therefore left multiplication by $u_n$.
As $u_n$ is an $n$-cycle, $L_{u_n}$ has order $n$.
\end{proof}

\begin{lemma}[the reflection]
\label{lem:reflection}
If every $g_k$ is an involution, then $L_{g_n}(C)=C$ and $L_{g_n}$ acts on $C$
as a reflection.
\end{lemma}

\begin{proof}
Write $w=w_1w_2\cdots w_N$ for $M_ng_n$, $N=n!$, so that $w_1\cdots w_{N-1}=M_n$
is a palindrome (Lemma~\ref{lem:palin}): $w_i=w_{N-i}$ for $1\le i\le N-1$. Let
$p_m=w_1\cdots w_m$ be the vertex at position $m$. By Lemma~\ref{lem:comp},
$p_{N-1}=g_n^{-1}$, so
\[
  p_{N-1-m}=p_{N-1}\bigl(w_{N-m}\cdots w_{N-1}\bigr)^{-1}
           =g_n^{-1}\bigl(w_m w_{m-1}\cdots w_1\bigr)^{-1}
           =g_n^{-1}\,w_1^{-1}\cdots w_m^{-1}.
\]
If all letters are involutions this is $g_n^{-1}p_m=g_n\,p_m$. Thus $L_{g_n}$
maps the vertex at position $m$ to the vertex at position $N-1-m$, an
orientation-reversing automorphism of $C$.
\end{proof}

\begin{theorem}[symmetry]
\label{thm:sym}
The cycle of a tower is dihedral if every $g_k$ is an involution, and cyclic
otherwise. In the first case $\Stab(C)=\langle L_{u_n},L_{g_n}\rangle\cong D_n$;
in the second $\Stab(C)=\langle L_{u_n}\rangle\cong C_n$. The all-involutive
towers are counted by
\[
  \prod_{k=2}^{n}2^{\lfloor (k-1)/2\rfloor}\bigl\lfloor (k-1)/2\bigr\rfloor!
  \;=\;2^{\lfloor (n-1)^2/4\rfloor}\prod_{k=2}^{n}\bigl\lfloor (k-1)/2\bigr\rfloor! .
\]
\end{theorem}

\begin{proof}
Lemma~\ref{lem:rotation} gives $\langle L_{u_n}\rangle\cong C_n$ inside
$\Stab(C)$, and Theorem~\ref{thm:rigidity} applied with $c=u_n$ gives
$|\Stab(C)|\le2n$. If every $g_k$ is an involution, Lemma~\ref{lem:reflection}
supplies an orientation-reversing element, so $\Stab(C)$ is dihedral of order
exactly $2n$, generated by $L_{u_n}$ and $L_{g_n}$.

Conversely, suppose some $L_h\in\Stab(C)$ reverses orientation, so that
$h\,p_m=p_{a-m}$ for some fixed $a$ and all $m$, where $p_m$ is the vertex at
position $m$ and $p_{m+1}=p_m w_{m+1}$. A left translation preserves the right
multiplier of an edge, since $\{x,xg\}\mapsto\{hx,hxg\}$. The edge of step
$m+1$, of multiplier $w_{m+1}$, is therefore carried to the edge joining
$h\,p_m=p_{a-m}$ to $h\,p_{m+1}=p_{a-m-1}$, whose multiplier read from
$p_{a-m}$ is $w_{a-m}^{-1}$. Hence
\[
  w_{m+1}=w_{a-m}^{-1}\qquad\text{for all }m .
\]
Both sides are generators; inversion preserves the layer, and there is exactly
one generator per layer, so $w_{a-m}^{-1}=w_{a-m}$. Thus every generator
occurring in the word is an involution, and all $n-1$ of them occur. So a
non-involutive tower admits no orientation-reversing element, and by
Theorem~\ref{thm:rigidity} its stabilizer is $\langle L_{u_n}\rangle\cong C_n$.

The count of all-involutive towers is Corollary~\ref{cor:count} below; the
closed form uses $\sum_{k=2}^{n}\lfloor (k-1)/2\rfloor=\lfloor (n-1)^2/4\rfloor$.
\end{proof}
Write $\nu(t)$ for the number of transpositions of an involution $t$. Recall
that an involution of $\Sym\{1,\dots,m\}$ has at most $\lfloor m/2\rfloor$ of
them, with equality exactly for the maximum matchings.

\begin{lemma}[involutive step]
\label{lem:invstep}
Let $k\ge3$ and let $h\in G_{k-1}\setminus G_{k-2}$ be an involution. The
$k$-cycles $c$ on $\{1,\dots,k\}$ for which $hc$ is again an involution are
exactly those satisfying $hch=c^{-1}$. Their number is $0$ unless
$\nu(h)=\lfloor (k-1)/2\rfloor$, in which case it equals
\[
  N_k=2^{\lfloor (k-1)/2\rfloor}\bigl\lfloor (k-1)/2\bigr\rfloor! .
\]
Moreover $hc$ then lies in $G_k\setminus G_{k-1}$ and satisfies
$\nu(hc)=\lfloor k/2\rfloor$.
\end{lemma}

\begin{proof}
Since $h^2=\id$, $(hc)^2=\id$ if and only if $hchc=\id$, that is $hch=c^{-1}$.

Fix a $k$-cycle $c$. The set of $x\in S_k$ with $xcx^{-1}=c^{-1}$ is a coset of
the centralizer $\langle c\rangle$, hence has $k$ elements, and for $k\ge3$ it
is disjoint from $\langle c\rangle$; the group $\langle c,x\rangle$ is dihedral
of order $2k$, so each of these $k$ elements is an involution. They are the
reflections of the $k$-gon whose cyclic order is given by $c$. Consequently a
solution $c$ exists only if $h$ is such a reflection, which constrains its
fixed points: one when $k$ is odd, none or two when $k$ is even. As $h$ fixes
$k$ and is an involution of $\{1,\dots,k-1\}$, it has $1$ fixed point in
$\{1,\dots,k\}$ when $k$ is odd and $2$ when $k$ is even precisely when
$\nu(h)=\lfloor (k-1)/2\rfloor$; otherwise it has strictly more fixed points and
no $c$ exists.

For the count, note that the number of admissible $c$ depends only on the cycle
type of $h$, because $S_k$ acts transitively by conjugation on each type and the
relation $hch=c^{-1}$ is equivariant. Count the pairs $(c,x)$ with $c$ a
$k$-cycle and $x$ an involution inverting it: there are $(k-1)!$ $k$-cycles,
each with $k$ such $x$, hence $k!$ pairs.

If $k$ is odd, every reflection of a $k$-gon has exactly one fixed point, and
the involutions of $S_k$ with one fixed point number $k\,(k-2)!!$; therefore
\[
  N_k=\frac{k!}{k\,(k-2)!!}=\frac{(k-1)!}{(k-2)!!}=(k-1)!!
      =2^{(k-1)/2}\Bigl(\frac{k-1}{2}\Bigr)! .
\]
If $k$ is even, exactly $k/2$ of the $k$ reflections of a $k$-gon have two fixed
points, so the pairs with $\nu(x)=(k-2)/2$ number $(k-1)!\,k/2=k!/2$, while the
involutions of $S_k$ with two fixed points number $\binom{k}{2}(k-3)!!$; therefore
\[
  N_k=\frac{k!/2}{\binom{k}{2}(k-3)!!}=\frac{(k-2)!}{(k-3)!!}=(k-2)!!
      =2^{(k-2)/2}\Bigl(\frac{k-2}{2}\Bigr)! .
\]
Both expressions equal $2^{\lfloor (k-1)/2\rfloor}\lfloor (k-1)/2\rfloor!$.

Finally $hc$ is a product of a reflection and a rotation of the same $k$-gon,
hence again a reflection; as $c$ is a rotation by one step, $hc$ lies in the
other reflection class when $k$ is even. In both parities $hc$ has the minimal
number of fixed points, that is $\nu(hc)=\lfloor k/2\rfloor$, and $hc$ moves $k$
because $c$ does and $h$ fixes $k$.
\end{proof}

\begin{corollary}
\label{cor:count}
The number of all-involutive towers is
\[
  \prod_{k=2}^{n}2^{\lfloor (k-1)/2\rfloor}\bigl\lfloor (k-1)/2\bigr\rfloor!
  \;=\;2^{\lfloor (n-1)^2/4\rfloor}\prod_{k=2}^{n}\bigl\lfloor (k-1)/2\bigr\rfloor! .
\]
\end{corollary}

\begin{proof}
Induction on $k$, the invariant being that $g_k$ is an involution of
$G_k\setminus G_{k-1}$ with $\nu(g_k)=\lfloor k/2\rfloor$. It holds at $k=2$
with $g_2=(1\,2)$. Lemma~\ref{lem:invstep} says that from such a $g_{k-1}$ there
are exactly $N_k$ choices of $u_k$ keeping $g_k=g_{k-1}u_k$ involutive, that all
of them preserve the invariant, and that no other $g_{k-1}$ admits any. Hence
the number of all-involutive towers is $\prod_{k=2}^{n}N_k$. The closed form
follows from $\sum_{k=2}^{n}\lfloor (k-1)/2\rfloor=\lfloor (n-1)^2/4\rfloor$.
\end{proof}

The resulting counts are $1,2,4,32,256,12288,589824$ for $n=2,\dots,8$, in
agreement with exhaustive enumeration.

\section{The pancake graph}
\label{sec:pancake}

Let $r_k$ denote the reversal of the first $k$ letters and
$P_n=\Cay(S_n,\{r_2,\dots,r_n\})$ the pancake graph.

\begin{corollary}
\label{cor:pancake}
Zaks' tower is the only tower inside $P_n$, and the genlex Hamiltonian cycles of
$P_n$ are exactly the $(n-1)!/2$ left translates of Zaks' cycle.
\end{corollary}

\begin{proof}
A tower inside $P_n$ must take $g_k\in G_k\setminus G_{k-1}$ among the prefix
reversals; the only prefix reversal moving $k$ and fixing $k+1,\dots,n$ is
$r_k$. So $g_k=r_k$ for all $k$, which is Zaks' tower. By
Theorem~\ref{thm:class} its cycle is the unique genlex Hamiltonian cycle of
$P_n$ through $\id$ with $n-1$ generators; its left translates give
$n!/|\Stab|=n!/2n=(n-1)!/2$ distinct cycles by Theorem~\ref{thm:sym}.
\end{proof}

\section{The extremes of Coxeter length}
\label{sec:coxeter}

A tower is an arbitrary choice of coset representative at each step of the
chain. The $G_k$, however, form a maximal chain of standard parabolic subgroups
of $W(A_{n-1})$, and this chain carries structure that a bare subgroup chain
does not: a length function, longest elements, Coxeter elements. Two towers are
singled out as the extremes of
\[
  \Lambda=\sum_{k=2}^{n}\ell(g_k).
\]

\begin{theorem}[longest]
\label{thm:long}
$\Lambda\le\binom{n+1}{3}$, with equality only for $g_k=w_{0,J_k}$, which is
Zaks' tower.
\end{theorem}

\begin{proof}
Since $g_k\in G_k=W_{J_k}$ we have $\ell(g_k)\le\ell(w_{0,J_k})=\binom{k}{2}$,
and $\sum_{k=2}^n\binom k2=\binom{n+1}{3}$. Equality forces
$g_k=w_{0,J_k}$ for every $k$. That this is a tower is the identity that
$u_k=w_{0,J_{k-1}}w_{0,J_k}$ is a Coxeter element of $W_{J_k}$, in particular a
$k$-cycle; and $w_{0,J_k}=r_k$, so this is Zaks' tower.
\end{proof}

For this tower $\ell(u_k)=k-1$ and $\ell(g_k)=\ell(g_{k-1})+\ell(u_k)$, so
$w_0=u_2u_3\cdots u_n$ is a reduced factorization: the staircase word. We
conjecture it is the only tower with that property.

\begin{theorem}[brick]
\label{thm:brick}
Let $\beta_k=s_{k-1}s_{k-3}\cdots$ be the product of one colour class of the
bipartite $2$-colouring of the diagram of $W_{J_k}$, so that $\beta_{k-1}$ is
the other class. Then $(\beta_2,\dots,\beta_n)$ is a tower,
$u_k=\beta_{k-1}\beta_k=c_-c_+$ is the bipartite Coxeter element of Steinberg's
factorization \cite{steinberg,coxeter}, every $\beta_k$ is an involution,
$\ell(\beta_k)=\lfloor k/2\rfloor$, and $\Lambda=\lfloor n^2/4\rfloor$.
\end{theorem}

\begin{proof}
The colour classes of a path diagram consist of pairwise commuting simple
reflections, so each $\beta_k$ is an involution and $\ell(\beta_k)$ equals the
number of factors, namely $\lfloor k/2\rfloor$. Their product
$u_k=\beta_{k-1}\beta_k$ is the bipartite Coxeter element of $W_{J_k}$, hence a
$k$-cycle; so $(\beta_k)$ is a tower. Finally
$\Lambda=\sum_{k=2}^n\lfloor k/2\rfloor=\lfloor n^2/4\rfloor$.
\end{proof}

\begin{conjecture}
\label{conj:min}
$\Lambda\ge\lfloor n^2/4\rfloor$ for every tower, with equality only for the
brick tower.
\end{conjecture}

The conjecture has been verified exhaustively for $n\le8$, a range comprising
$125\,411\,328\,000$ towers at $n=8$; throughout it the minimum of $\Lambda$ is
$\lfloor n^2/4\rfloor$ and is attained by the brick tower alone.

\begin{figure}[t]
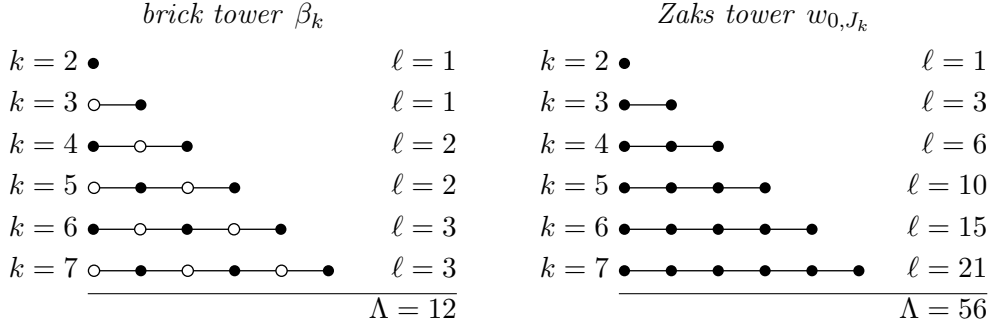

\centering
\begin{tabular}{@{}r@{\ }l@{\hspace{1.2em}}r@{}}
\multicolumn{3}{c}{\emph{brick tower }$\beta_k$}\\[4pt]
$k=2$ & \coxrow{2}{1} & $\ell=1$\\[2pt]
$k=3$ & \coxrow{3}{1} & $\ell=1$\\[2pt]
$k=4$ & \coxrow{4}{1} & $\ell=2$\\[2pt]
$k=5$ & \coxrow{5}{1} & $\ell=2$\\[2pt]
$k=6$ & \coxrow{6}{1} & $\ell=3$\\[2pt]
$k=7$ & \coxrow{7}{1} & $\ell=3$\\[2pt]
\cline{2-3}
      &               & $\Lambda=12$
\end{tabular}
\hspace{2.2em}
\begin{tabular}{@{}r@{\ }l@{\hspace{1.2em}}r@{}}
\multicolumn{3}{c}{\emph{Zaks tower }$w_{0,J_k}$}\\[4pt]
$k=2$ & \coxrow{2}{0} & $\ell=1$\\[2pt]
$k=3$ & \coxrow{3}{0} & $\ell=3$\\[2pt]
$k=4$ & \coxrow{4}{0} & $\ell=6$\\[2pt]
$k=5$ & \coxrow{5}{0} & $\ell=10$\\[2pt]
$k=6$ & \coxrow{6}{0} & $\ell=15$\\[2pt]
$k=7$ & \coxrow{7}{0} & $\ell=21$\\[2pt]
\cline{2-3}
      &               & $\Lambda=56$
\end{tabular}
\caption{The two towers, $n=7$. Row $k$ is the Coxeter diagram of $W_{J_k}$,
whose nodes are $s_1,\dots,s_{k-1}$; a node is filled when that reflection
occurs in $g_k$.}
\end{figure}

The condition of Theorem~\ref{thm:class} is weaker than requiring $u_k$ to be
a Coxeter element: there are $2^{k-2}$ Coxeter elements among the $(k-1)!$
$k$-cycles. Both extremal towers nevertheless lie in the Coxeter class, and
both are all-involutive, hence dihedral. Since the Coxeter number is $h=n$ here,
\[
  \Stab(C)=\langle c,\,w_0\rangle\quad\text{and}\quad
  \Stab(C)=\langle c_-,\,c_+\rangle,
\]
the second being the dihedral group of order $2h$ acting in the Coxeter plane
\cite{coxeter}. Both towers factor each $u_k$ into two involutions; they are the
two extremes in length of doing so along the chain.

The quotients $u_k$ may be chosen independently --- at each level, either the
rotation of Zaks' tower or the bipartite factor --- but length and involutivity
are properties of $g_k=g_{k-1}u_k$, hence of the entire history of choices.
Among the $2^{n-2}$ hybrid towers obtained by mixing the two, only the two pure
towers are all-involutive; every mixed tower is therefore cyclic.

\section{Open problems}

\begin{question}
Is Conjecture~\ref{conj:min} true? The difficulty is that the bound is not
pointwise, unlike the maximum, where $\ell(g_k)\le\ell(w_{0,J_k})$ settles
everything level by level. For $n=6$ the tower $g_2=(1\,2)$, $g_3=(2\,3)$,
$g_4=(1\,3\,4)$, $g_5=(1\,2\,4\,5\,3)$, $g_6=(5\,6)$ has $\ell(g_6)=1$, far
below $\lfloor 6/2\rfloor$, and $\Lambda=11$: the shortfall is paid one level
down.
\end{question}

\begin{question}
For a chain $\{1\}=G_1<\cdots<G_N=G$ of finite groups with
$m_k=[G_k:G_{k-1}]$, the recursion $M_k=(M_{k-1}g_k)^{m_k-1}M_{k-1}$ lists $G$
with all cosets contiguous as soon as $\langle u_k\rangle$ acts simply
transitively on $G_k/G_{k-1}$ --- the non-normal analogue of a polycyclic
series --- and the symmetry statement survives. This recovers the reflected Gray
code, Knuth's mixed-radix codes \cite[\S7.2.1.1]{knuth4a} and Tompkins--Paige
\cite{tompkins}. Does everything above transport to type $B$, that is, to signed
permutations and burnt pancakes \cite{sawadawilliams1}? The numerology is that
$[W_{J_k}:W_{J_{k-1}}]$ must be met by an element of that order: it equals the
Coxeter number of $W_{J_k}$ in types $A$ and $B$ but not in type $D$.
\end{question}

\begin{question}
The all-involutive towers produce
$2^{\lfloor(n-1)^2/4\rfloor}\prod_{k=2}^{n}\lfloor(k-1)/2\rfloor!$ Hamiltonian
cycles on $S_n$ whose left-regular stabilizer is dihedral of order $2n$. Are there others, and must
the order be $2n$? The question is meant with both hypotheses dropped: over all
Cayley graphs of $S_n$ and all Hamiltonian cycles in them, which carry a
dihedral group of left translations?
\end{question}



\begin{thebibliography}{99}

\bibitem{zaks} S. Zaks, \emph{A new algorithm for generation of permutations},
BIT 24 (1984) 196--204.

\bibitem{sawadawilliams1} J. Sawada, A. Williams, \emph{Greedy flipping of
pancakes and burnt pancakes}, Discrete Appl. Math. 210 (2016) 61--74.

\bibitem{walsh} T. Walsh, \emph{Generating Gray codes in $O(1)$ worst-case time
per word}, DMTCS 2003, LNCS 2731, 73--88.

\bibitem{knuth4a} D. E. Knuth, \emph{The Art of Computer Programming}, vol. 4A,
Addison--Wesley, 2011, \S7.2.1.1--7.2.1.2 and \S7.2.1.7.

\bibitem{sims} C. C. Sims, \emph{Computational methods in the study of
permutation groups}, in: Computational Problems in Abstract Algebra (J. Leech,
ed.), Pergamon, Oxford, 1970, 169--183.

\bibitem{ordsmith} R. J. Ord-Smith, \emph{Generation of permutation sequences},
Comm. ACM 10 (1967) 452; 12 (1969) 638; Comp. J. 14 (1971) 136--139.

\bibitem{ehrlich} G. Ehrlich, \emph{Loopless algorithms for generating
permutations, combinations, and other combinatorial configurations}, J. ACM 20
(1973) 500--513.

\bibitem{kokosinski} Z. Kokosi\'nski, \emph{On generation of permutations
through decomposition of symmetric groups into cosets}, BIT 30 (1990) 583--591.

\bibitem{mutze} T. M\"utze, \emph{Combinatorial Gray codes --- an updated
survey}, Electron. J. Combin. DS26 (2023), arXiv:2202.01280.

\bibitem{tompkins} C. Tompkins, \emph{Machine attacks on problems whose
variables are permutations}, Proc. Sympos. Appl. Math. 6 (1956) 195--211.

\bibitem{coxeter} H. S. M. Coxeter, \emph{The product of the generators of a
finite group generated by reflections}, Duke Math. J. 18 (1951) 765--782.

\bibitem{steinberg} R. Steinberg, \emph{Finite reflection groups}, Trans. Amer.
Math. Soc. 91 (1959) 493--504.

\bibitem{gmm} P. Gregor, A. Merino, T. M\"utze, \emph{The Hamilton compression
of highly symmetric graphs}, Ann. Comb. 28 (2024) 379--437.

\bibitem{sharvitbenjo} Y. Sharvit, O.-H. Benjo, \emph{The left-regular
stabilizer of Zaks' Hamiltonian cycle in the pancake graph}, arXiv:2607.04658,
2026.

\bibitem{abulafia} Y. Sharvit, \emph{The Abulafia graph: the order of the
tseruf, Zaks' formula, and the emergence of caustics}, Zenodo, 2026,
\texttt{doi:10.5281/zenodo.21227448}.

\end{thebibliography}
\end{document}